\documentclass[12pt,a4paper]{article}
\usepackage{amsmath,amssymb,amsbsy,amscd,amsthm}
\usepackage{multirow}
\newtheorem{theorem}{Theorem}[section]
\newtheorem{proposition}[theorem]{Proposition}
\newtheorem{lemma}[theorem]{Lemma}
\newtheorem{corollary}[theorem]{Corollary}
\newtheorem{remark}[theorem]{Remark}

\newcommand{\Kx }{K[\x ]}

\newcommand{\zs}{\{ 0\} }

\newtheorem{problem}[theorem]{Problem}
\newcommand{\x }{{\bf x}}

\newcommand{\Rx }{R[\x ]}
\newcommand{\E }{\mathop{\rm E}\nolimits}

\newcommand{\T }{\mathop{\rm T}\nolimits}
\newcommand{\Aut }{\mathop{\rm Aut}\nolimits}
\newcommand{\Aff }{\mathop{\rm Aff}\nolimits}
\newcommand{\Add }{\mathop{\rm Add}\nolimits}
\newcommand{\BA }{\mathop{\rm BA}\nolimits}
\newcommand{\GL}{\mathop{\rm GL}\nolimits}
\newcommand{\Spec}{\mathop{\rm Spec}\nolimits}

\newcommand{\ca}{\mathcal{A}}
\newcommand{\cb}{\mathcal{B}}

\newcommand{\A}{\ensuremath{\mathbb A}}

\newcommand{\N}{\ensuremath{\mathbb N}}

\newcommand{\Z}{\ensuremath{\mathbb Z}}

\newcommand{\pri}{\ensuremath{\smallsetminus}}
\newcommand{\sm}{\setminus}

\newcommand{\al}{\ensuremath{\alpha}}

\begin{document}

\title
{Generalisations of the tame automorphisms over a domain of positive characteristic}

\author{Eric Edo\and Shigeru Kuroda}
\date{}


\maketitle

\begin{abstract}
In this paper, we introduce two generalizations of the tame subgroup of the automorphism group
of a polynomial ring over a domain of positive characteristic.
We study detailed structures of these new `tame subgroups' in the case of two variables.
\end{abstract}

\section{Introduction}

{\bf A. The Tame Generator Problem}\\

Throughout this paper, 
let $R$ be a domain of characteristic $p\geq 0$,
and $\Rx :=R[x_1,\ldots ,x_n]$
the polynomial ring in $n$ variables over $R$,
where $n\in \N $ (when $n=2$, we set $x=x_1$ and $y=x_2$).
We discuss the structure of the group
$\Aut _R\Rx $ of automorphisms
of the $R$-algebra $\Rx $.
Since each $\phi \in \Aut _R\Rx $
is uniquely determined by $\phi (x_i)$
for $i=1,\ldots ,n$,
we sometimes identify $\phi $
with the $n$-tuple $(\phi (x_1),\ldots ,\phi (x_n))$
of elements of $\Rx $.
We say that $\phi $ is {\it elementary}
if there exist $l\in\{1,\ldots,n\}$,
$a\in R^{\times }$ and
$$
f\in A_l:=R[x_1,\ldots ,x_{l-1},x_{l+1},\ldots ,x_n]
$$
such that
$\phi (x_l)=ax_l+f$ and $\phi (x_i)=x_i$ for all $i\neq l$.
Since $R$ is a domain,
$\phi $ is elementary
if and only if $\phi $ belongs to $\Aut _{A_l}\Rx $
for some $l\in\{1,\ldots,n\}$.
We say that $\phi $ is {\it affine}
if there exist $M\in \GL _n(R)$
and $(b_1,\ldots ,b_n)\in R^n$ such that
$$
(\phi (x_1),\ldots ,\phi (x_n))
=(x_1,\ldots ,x_n)M+(b_1,\ldots ,b_n).
$$
If $b_1=\cdots =b_n=0$,
we call $\phi $ a {\it linear automorphism}.
If $M$ is the identity matrix,
we call $\phi $ a {\it translation}.
We define the
{\it affine subgroup} $\Aff _n(R)$,
{\it elementary subgroup} $\E _n(R)$,
and {\it tame subgroup} $\T _n(R)$ of $\Aut _R\Rx $
to be
the subgroups of $\Aut _R\Rx $ generated by
all the affine automorphisms,
all the elementary automorphisms,
and $\Aff _n(R)\cup \E _n(R)$,
respectively.
We say that $\phi \in \Aut _R\Rx $
is {\it tame} if $\phi $ belongs to $\T _n(R)$,
and {\it wild} otherwise.

It is natural to consider the following problem.

\begin{problem}\label{problem:tgp}\rm
Does it hold that $\Aut _R\Rx =\T _n(R)$?
\end{problem}

This problem,
called the {\it Tame Generators Problem},
is one of the most famous problems
in Polynomial Ring Theory.
Since $R$ is a domain,
every element of $\Aut _RR[x_1]$
is affine and elementary.
Hence,
the equality holds for $n=1$.
When $n=2$ and $R$ is a field,
the answer to Problem~\ref{problem:tgp} is affirmative
due to Jung~\cite{Jung} and van der Kulk~\cite{Kulk}.
On the other hand,
if $n=2$ and $R$ is not a field,
then there always exists a wild automorphism due to
Nagata~\cite{Nagata}.
Nagata also conjectured that $\Aut _R\Rx \neq \T _3(R)$
if $n=3$ and $R$ is a field,
and gave a candidate for a wild automorphism.
The conjecture was recently solved in the affirmative
by Shestakov-Umirbaev~\cite{SU} in the case of $p=0$.
Problem~\ref{problem:tgp} remains open in the other cases.

Since $\Aut _R\Rx $ is not equal to $\T _n(R)$ in general,
our next interest is to find a `better'
candidate for a generating set for the group $\Aut _R\Rx $,
and generalize the notion of tame automorphisms.
The purpose of this paper is to introduce
two generalizations of tame automorphisms
in the case of $p>0$,
which are both obtained
by generalizing the notion of affine automorphisms.
We study the detailed structures of these new `tame subgroups'
when $n=2$.\\

{\bf B. The geometrically affine subgroup}\\

Recall that $f\in \Rx $ is said to be {\it additive} if
$$
f(y_1+z_1,\ldots ,y_n+z_n)
=f(y_1,\ldots ,y_n)+f(z_1,\ldots ,z_n),
$$
where $y_1,\ldots ,y_n$ and $z_1,\ldots ,z_n$
are indeterminates over $R$.
If $p=0$,
then $f$ is additive if and only if
$f$ is a linear form over $R$.
If $p>0$,
then $f$ is additive if and only if
$f$ is a $p$-{\it polynomial}, i.e.,
a polynomial of the form
$$
f=\sum _{i=0}^n
\sum _{j\geq 0}a_{i,j}x_i^{p^j}
$$
for some $a_{i,j}\in R$ for $i=1,\ldots ,n$ and $j\geq 0$. 
Let $\Add _n(R)$ be the set
of $\phi \in \Aut _R\Rx $
such that $\phi (x_1),\ldots ,\phi (x_n)$
are additive polynomials.
Then,
$\Add _n(R)$ forms a subgroup of $\Aut _R\Rx $.
Actually,
if $f,g_1,\ldots ,g_n\in \Rx $
are additive,
then $f(g_1,\ldots ,g_n)$ is also additive.
Note that $\phi \in \Aut _R\Rx $ belongs to $\Add _n(R)$
if $\phi $ is linear,
where the converse also holds when $p=0$.
If $p>0$ and $n\geq 2$, however,
$\Add _n(R)$ contains non-linear automorphisms.
The second author~\cite{Kuroda} pointed out that,
if $R$ is not a field,
then $\Add _2(R)$ is not contained in $\T _2(R)$
for any $p>0$.
He also showed that,
if $R$ is a field,
then $\Add _n(R)$ is contained in $\T _n(R)$
for any $n\geq 1$
(see also \cite{Tanaka}).

Geometrically,
$\Add _n(R)$ can be explained as follows.
Recall that $\Aut _R\Rx $
is identified with the automorphism group
of the affine $n$-space $\A ^n:=\Spec \Rx $ over $R$.
We may regard $\A ^n$ also as
the $n$-dimensional {\it vector group},
i.e.,
the affine algebraic group scheme with
the coproduct $\mu :\Rx \to \Rx \otimes _R\Rx $,
coidentity $\epsilon :\Rx \to R$,
and coinverse $\iota :\Rx \to \Rx $ defined by
$$
\mu (x_i)=x_i\otimes 1+1\otimes x_i,\quad
\epsilon (x_i)=0\quad{\rm and}\quad
\iota (x_i)=-x_i
$$
for $i=1,\ldots ,n$,
respectively.
Then,
$\Add _n(R)$ coincides with the automorphism
group of the $n$-dimensional vector group.
From this point of view,
$\Add _n(R)$ is regarded as a generalization
of the group of linear automorphisms.

The groups of all the translations
and all the linear automorphisms
of $\Rx $ over $R$
can be canonically identified with
the groups $R^n$ and $\GL _n(R)$,
respectively.
Then, we have
$\Aff _n(R)=R^n\rtimes \GL _n(R)$.
Since $\Add _n(R)$ is a generalization of $\GL _n(R)$,
we may consider
$$
\Aff ^{\rm g}_n(R):=\langle R^n, \Add _n(R)\rangle
$$
as a generalization of $\Aff _n(R)$.
Here,
for a group $G$
and subgroups $H_1,\ldots ,H_r$ of $G$,
we denote by
$\langle H_1,\ldots ,H_r\rangle $
the subgroup of $G$ generated by $H_1\cup \cdots \cup H_r$.
We call $\Aff ^{\rm g}_n(R)$
the {\it geometrically affine subgroup} of $\Aut _R\Rx $.
As in the case of the affine subgroup,
it holds that
$$
\Aff ^{\rm g}_n(R)=R^n\rtimes \Add _n(R).
$$
Actually,
for $\phi =(f_1,\ldots ,f_n)\in \Add _n(R)$ 
and $\tau =(x_1+b_1,\ldots ,x_n+b_n)\in R^n$, 
we have 
$$
\phi ^{-1}(\tau (\phi (x_i)))
=\phi ^{-1}\bigl(f_i(x_1,\ldots ,x_n)+
f_i(b_1,\ldots ,b_n)\bigr)
=x_i+f_i(b_1,\ldots ,b_n)
$$
for $i=1,\ldots ,n$. 
Hence,
$\phi \in \Aut _R\Rx $ belongs to $\Aff ^{\rm g}_n(R)$
if and only if there exists $(b_1,\ldots ,b_n)\in R^n$
such that $\phi (x_i)-b_i$ is an additive polynomial
for $i=1,\ldots ,n$.

If $p=0$,
then we have $\Aff ^{\rm g}_n(R)=\Aff _n(R)$
for any $n\geq 1$,
since $\Add _n(R)=\GL _n(R)$.
If $R$ is a field,
then $\Aff ^{\rm g}_n(R)$ is contained in $\T _n(R)$
for any $n\geq 1$ and $p\geq 0$,
since so is $\Add _n(R)$.
If $R$ is not a field and $p>0$,
then $\Aff ^{\rm g}_2(R)$ is not contained in $\T _2(R)$,
since neither is $\Add _2(R)$.
So
we define
the {\it geometrically tame subgroup}
$\T^{\rm g}_n(R)$ of $\Aut _R\Rx $ by
$$
\T^{\rm g}_n(R):=\langle \Aff ^{\rm g}_n(R),\T_n(R)\rangle ,
$$
which is a natural generalization of $\T _n(R)$.\\

{\bf C. The differentially affine subgroup}\\

Let us look at another aspect of affine automorphisms.
For each $\phi \in \Aut _R\Rx $,
we denote by $J\phi $ the Jacobian matrix of $\phi $.
Then,
$J\phi $ always belongs to $\GL _n(\Rx )$.
Note that $J\phi $ belongs to
$M_n(R)\cap \GL _n(\Rx )=\GL _n(R)$
if $\phi $ is affine,
where the converse also holds when $p=0$.
We define
$$
\Aff ^{\rm d}_n(R)
:=\{ \phi \in \Aut _R\Rx \mid
J\phi \in \GL _n(R)\} .
$$
Then,
$\Aff ^{\rm d}_n(R)$ contains $\Aff _n(R)$,
and $\Aff ^{\rm d}_n(R)=\Aff _n(R)$ if $p=0$.
By chain rule,
$$
J(\phi \circ \psi )=\phi (J\psi )J\phi
$$
holds for each $\phi ,\psi \in \Aut _R\Rx $,
where
$\phi (J\psi )$ is the matrix
obtained from $J\psi $
by mapping each component by $\phi $.
Since $\phi (J\psi )=J\psi $
if $\psi $ belongs to $\Aff ^{\rm d}_n(R)$,
we see that
$\Aff ^{\rm d}_n(R)$ forms a subgroup of $\Aut _R\Rx $.
We call $\Aff ^{\rm d}_n(R)$
the {\it differentially affine subgroup}
of $\Aut _R\Rx $.
By definition,
$\Aff ^{\rm d}_n(R)$ is equal
to the inverse image of $\GL _n(R)$
by the map
$$
\Aut _R\Rx \ni \phi \mapsto J\phi \in \GL _n(\Rx ).
$$
This map itself is not a homomorphism of groups if $n\geq 2$,
but the induced map
$\Aff ^{\rm d}_n(R)\to \GL _n(R)$
is a homomorphism
of groups for any $n\geq 1$.

Assume that $p>0$. Then,
$\phi \in \Aut _R\Rx $ belongs to $\Aff ^{\rm d}_n(R)$
if and only if
there exists a linear automorphism $\psi $
such that $J\phi =J\psi $,
or equivalently
$$
\phi (x_i)-\psi (x_i)
\in\left\{
f\in \Rx \Bigm|
\frac{\partial f}{\partial x_j}=0
{\rm \ for\ }j=1,\ldots ,n
\right\}
=R[x_1^p,\ldots ,x_n^p]
$$
for $i=1,\ldots ,n$.
From this,
we see that $\Aff ^{\rm d}_n(R)$ contains $\Add ^{\rm g}_n(R)$,
and thus contains $\Aff ^{\rm g}_n(R)$.
We define the {\it differentially tame subgroup}
$\T^{\rm d}_n(R)$ of $\Aut _R\Rx $ by
$$
\T^{\rm d}_n(R):=\langle \Aff ^{\rm d}_n(R),\T_n(R)\rangle .
$$

{\bf D. Affine type subgroups}\\

It holds that
\begin{equation}\label{eq:ABC}
\T _n(R)
\stackrel{\rm a}{\subset }\T ^{\rm g}_n(R)
\stackrel{\rm b}{\subset }\T ^{\rm d}_n(R)
\stackrel{\rm c}{\subset }\Aut _R\Rx .
\end{equation}
More precisely,
we have the following statements 
(cf.\ Table 1):

\noindent (1)
If $n=1$,
or if $n=2$ and $R$ is a field,
then the four subgroups above are equal,
since $\T _n(R)=\Aut _R\Rx $.

\noindent (2)
If $p=0$,
then we have
$\T _n(R)=\T ^{\rm g}_n(R)=\T ^{\rm d}_n(R)$
for any $n\geq 1$,
since $\Aff _n(R)=\Aff ^{\rm g}_n(R)=\Aff ^{\rm d}_n(R)$.

\noindent (3)
If $R$ is a field,
then we have $\T_n(R)=\T ^{\rm g}_n(R)$
for any $n\geq 1$,
since $\Add _n(R)$ is contained in $\T _n(R)$.

\noindent (4)
If $p>0$ and $R$ is not a field,
then we have $\T_2(R)\neq \T ^{\rm g}_2(R)$,
since $\Add _2(R)$ is not contained in $\T _2(R)$.

\medskip 

\begin{center}

Known facts about (a,b,c) for a, b and c in (\ref{eq:ABC})

\smallskip 

\begin{tabular}{|c||c|c|c|c|} \hline 
 &
\multicolumn{2}{|c|}{$p=0$} &
\multicolumn{2}{|c|}{$p>0$} \\ \hline 
$R$ & field & non-field & field & non-field \\ 
\hline 
$n=1$ & \multicolumn{4}{|c|}{ $(=,=,=)$ } \\
\cline{1-1} \cline{3-3} \cline{5-5} 
$n=2$ & & $(=,=,\neq )$ & & $(\neq ,\ ?,\ ?)$ \\ \hline 
$n=3$ & \multicolumn{2}{|c|}{$(=,=,\neq )$ } & 
\multirow{2}{*}{$(=,\ ?,\ ?)$} &  
\multirow{2}{*}{(\ ?,\ ?,\ ?)} \\ \cline{1-3} 
$n\geqslant 4$ & \multicolumn{2}{|c|}{$(=,=,\ ?)$} & & \\ \hline 
\end{tabular}

\medskip 

Table 1
\end{center}

\medskip

Similarly to Problem~\ref{problem:tgp},
we can consider the following problem.

\begin{problem}\label{problem:gtgp}\rm
Let $R$ be a domain of positive characteristic.

\noindent{\rm (1)} Does it hold that
$\T ^{\rm g}_n(R)=\T ^{\rm d}_n(R)$
when $n\geq 2$ and $R$ is not a field?

\noindent{\rm (2)} Does it hold that
$\T ^{\rm d}_2(R)=\Aut _RR[x_1,x_2]$
when $R$ is not a field?

\noindent{\rm (3)} Does it hold that
$\T ^{\rm d}_n(R)=\Aut _R\Rx $ when $n\geq 3$?
\end{problem}

The following theorem
is the first main result of this paper.
This gives negative solutions to
Problem~\ref{problem:gtgp} (1) for $n=2$
and Problem~\ref{problem:gtgp} (2).
Here, 
$H\ntriangleleft G$ denotes that $H$ 
is not a normal subgroup of a group $G$.

\begin{theorem}\label{thm:gdt}
Let $R$ be a domain of positive characteristic
which is not a field.
Then,
we have
$$
\T_2(R)
\ntriangleleft
\T ^{\rm g}_2(R)
\ntriangleleft
\T ^{\rm d}_2(R)
\ntriangleleft
\Aut _RR[x_1,x_2].
$$
In particular, 
$\T_2(R)$, 
$\T ^{\rm g}_2(R)$, 
$\T ^{\rm d}_2(R)$ 
and $\Aut _RR[x_1,x_2]$ 
are different. 
\end{theorem}

Let $\BA _n(R)$ be the {\it de Jonqui\`eres subgroup} 
of $\Aut _R\Rx $, i.e., 
the set of $\phi \in \Aut _R\Rx $ such that 
$\phi (x_i)$ belongs to $R[x_i,\ldots ,x_n]$ 
for $i=1,\ldots ,n$. 
Let $K$ be the field of fractions of $R$. 
We remark that 
$\Aff ^{\rm g}_2(K)$ and $\Aff ^{\rm d}_2(K)$ 
share the following two properties.
First, 
if we denote by ${\cal H}$ one of these two groups, 
we have (see Proposition~\ref{prop:gdaffine}):
$${\cal H}=\langle \Aff _2(K),\BA _2(K)\cap {\cal H}\rangle.$$
The second property is more technical
and say roughly speaking that the intersection $\BA _2(K)\cap {\cal H}$ 
can be described by monomials of a certain type
(see Section~4 for a precise definition). We say that a subgroup ${\cal H}$ is an \textit{affine type subgroup} of $\Aut _KK[x_1,x_2]$
if it satisfied this two properties. With this definition $\Aff_2(K)$, $\Aff ^{\rm g}_2(K)$, $\Aff ^{\rm d}_2(K)$ and $\Aut _KK[x_1,x_2]$
are affine type subgroups of $\Aut _KK[x_1,x_2]$ but there exist infinitely many other affine type subgroup of $\Aut _KK[x_1,x_2]$.
For a given affine type subgroup ${\cal H}$ of $\Aut _KK[x_1,x_2]$ we define
a  subgroup of $\Aut _RR[x_1,x_2]$ by
$$T({\cal H})=\langle{\cal H}\cap\Aut _RR[x_1,x_2],T_2(R)\rangle.$$
For example $T(\Aff_2(K))=T_2(R)$, $T(\Aff^{\rm g}_2(K))=T_2^{\rm g}(R)$, $T(\Aff^{\rm d}_2(K))=T_2^{\rm d}(R)$
and $T(\Aut _KK[x_1,x_2])=\Aut _RR[x_1,x_2]$. Finally we prove the following general result (see Theorem~\ref{thm:notnormal}).

\begin{theorem}\label{thm:main1}
We assume that $R$ is not a field.
Let ${\cal H}_1,{\cal H}_2$ be two affine type subgroups of $\Aut _KK[x_1,x_2]$ such that ${\cal H}_1\subsetneqq{\cal H}_2$.
Then $T({\cal H}_1)$ is a subgroup of $T({\cal H}_2)$ which is not normal.
\end{theorem}

The second main result of this paper
is about the group theoretic-structures of
$\T ^{\rm g}_2(R)$ and $\T ^{\rm d}_2(R)$ and more generally $T({\cal H})$ where ${\cal H}$ is an affine type subgroup of $\Aut _KK[x_1,x_2]$.
For a group $G$ and subgroups $H$ and $H'$ of $G$
with $G=\langle H,H'\rangle $,
recall that $G$ is said to be the
{\it amalgamated product} 
of $H$ and $H'$ over $H\cap H'$ (we write $G=H*_{\cap}H'$)
if $\al_1\cdots \al_l\neq 1$ for any $l\geq 1$, and $\al_i\in (H\pri H')\cup (H'\pri H)$ for $i=1,\ldots ,l$ such that
$\al_i\al_{i+1}\not\in H\cup H'$ for $i=1,\ldots ,l-1$.
Then, the following theorem is well known 
(cf.\cite{van den Essen} Corollary~5.1.3).

\begin{theorem}\label{thm:afp}
If $R$ is a domain, then $\T _2(R)=\Aff _2(R)*_{\cap}\BA _2(R)$. 
\end{theorem}

Using Theorem~\ref{thm:afp}, we prove the following result (see Theorem~\ref{thm:productaffine}).

\begin{theorem}\label{thm:main2}
Let ${\cal H}$ be an affine type subgroup of $\Aut _KK[x_1,x_2]$. Then
$$T({\cal H})=({\cal H}\cap \Aut _RR[x_1,x_2])*_{\cap}{\rm T}_2(R).$$
In particular, $\T ^{\rm g}_2(R)=\Aff ^{\rm g}_2(R)*_{\cap}\T _2(R)$ 
and $\T ^{\rm d}_2(R)=\Aff ^{\rm d}_2(R)*_{\cap}\T _2(R)$.
\end{theorem}

The structure of this paper is as follows.
In Section~2, we prove that $${\cal H}=\langle \Aff _2(K),\BA _2(K)\cap {\cal H}\rangle.$$
for the subgroups ${\cal H}=\Aff ^{\rm g}_2(K)$ and ${\cal H}=\Aff ^{\rm d}_2(K)$. In the geometrical case,
we prove this property in any dimension (see Proposition~\ref{prop:gdaffine}).
In Section~3, we prove a technical result on abstract groups with amalgamated structure 
(see Theorem~\ref{thm}). We deduce three corollaries. 
In section~4, we introduce the notion of affine type subgroup of $\Aut _KK[x_1,x_2]$
and we apply the first two corollaries of Section~3 to obtain Theorems~\ref{thm:main1} and \ref{thm:main2}.
The third one is used to study in detail the case of length~$3$ differentially tame automorphisms
(see Corollary~\ref{ex}).

\section{Geometrically affine and differentially affine automorphisms}
\setcounter{equation}{0}

Throughout this section,
we assume that $K$ is a field of characteristic $p>0$.
The goal of this section
is to prove the following proposition.

\begin{proposition}\label{prop:gdaffine}
\noindent{\rm (i)} For each $n\geq 1$,
we have
$$
\Aff ^{\rm g}_n(K)
=\langle \Aff _n(K), \BA _n(K)\cap \Aff ^{\rm g}_n(K)\rangle .
$$

\noindent{\rm (ii)} For $n=1,2$,
we have
$$
\Aff ^{\rm d}_n(K)
=\langle \Aff _n(K),\BA _n(K)\cap \Aff ^{\rm d}_n(K)\rangle .
$$
\end{proposition}

For $i=1,\ldots ,n$,
we denote by $\E _n^i(K)$
the set of automorphism of $\Kx $
over $K[x_1,\ldots ,x_{i-1},x_{i+1},\ldots ,x_n]$.
Then,
$E_n^1(K)$ is contained in $\BA _n(K)$.

We begin with the following useful lemma.

\begin{lemma}\label{lem:compatibility}
Let $\ca $ be a subgroup of $\Aut _K\Kx $
containing $\Aff _n(K)$.
If
\begin{equation}\label{eq:lem:compatibility}
\ca \subset \left\langle
\Aff _n(K),
\E ^1_n(K)\cap \ca ,\ldots ,\E ^n_n(K)\cap \ca
\right\rangle ,
\end{equation}
then we have
$\ca =\langle \Aff _n(K),\BA _n(K)\cap \ca \rangle $.
\end{lemma}
\begin{proof}\rm
Since $\E ^1_n(K)$ is contained in $\BA _n(K)$,
and
$\Aff _n(K)$ is contained in $\ca $
by assumption,
we have
$$
\cb :=\langle \Aff _n(K),\E ^1_n(K)\cap \ca \rangle
\subset \langle \Aff _n(K),\BA _n(K)\cap \ca \rangle
\subset \ca .
$$
Hence,
it suffices to show that $\ca $ is contained in $\cb $.
For $i=1,\ldots ,n$,
we define $\tau _i\in \Aff _n(K)$ by
$$
\tau _i(x_1)=x_i,\quad
\tau _i(x_i)=x_1\quad
{\rm\ and\ }\quad
\tau _i(x_j)=x_j{\rm\ for\ }j=1,\ldots ,n{\rm\ with\ }j\neq 1,i.
$$
Then,
we have $\tau _i\circ \E ^1_n(K)\circ \tau _i=\E ^i_n(K)$.
Since $\tau _i$ belongs to $\Aff _n(K)\subset \ca $,
it follows that
$$
\E ^i_n(K)\cap \ca
=(\tau _i\circ \E ^1_n(K)\circ \tau _i)
\cap (\tau _i\circ \ca \circ \tau _i)
=\tau _i\circ (\E ^1_n(K)\cap \ca )\circ \tau _i.
$$
Hence,
$\E ^i_n(K)\cap \ca $
is contained in $\cb $ for $i=1,\ldots ,n$.
By definition,
$\Aff _n(K)$ is contained in $\cb $.
Therefore,
$\ca $ is contained in $\cb $ by
(\ref{eq:lem:compatibility}).
\end{proof}

Let us prove Proposition~\ref{prop:gdaffine} (i).
Since $\Aff ^{\rm g}_n(K)$ contains $\Aff _n(K)$,
we have only to check
that (\ref{eq:lem:compatibility})
holds for $\ca :=\Aff ^{\rm g}_n(K)$
thanks to Theorem~\ref{lem:compatibility}.
By the definition of $\Aff ^{\rm g}_n(K)$,
it suffices to show that
$\Add _n(K)$ is contained in the
the right-hand side of (\ref{eq:lem:compatibility}).
This follows from the following theorem.

\begin{theorem}[{\cite[Corollary 2.3]{Kuroda}}]\label{thm:vectgp}
For each $n\geq 1$,
the group $\Add _n(K)$
is generated by elementary automorphisms
belonging to $\Add _n(K)$.
\end{theorem}

In fact,
since
$\E _n^i(K)\cap \Aff ^{\rm g}_n(K)$
contains $\E _n^i(K)\cap \Add _n(K)$
for each $i$,
the right-hand side of (\ref{eq:lem:compatibility})
contains
$$
\langle
\E _n^1(K)\cap \Add _n(K),\ldots ,
\E _n^n(K)\cap \Add _n(K)
\rangle ,
$$
and hence contains $\Add _n(K)$
by Theorem~\ref{thm:vectgp}.

Next, we show Proposition~\ref{prop:gdaffine} (ii).
If $n=1$,
then we have $\Aut _K\Kx =\Aff _1(K)$.
Since $\Aff _1(K)\subset \Aff ^{\rm d}_1(K)\subset \Aut _K\Kx $,
it follows that $\Aff ^{\rm d}_1(K)=\Aff _1(K)$.
Thus,
we know that
$$
\Aff ^{\rm d}_1(K)=
\langle
\Aff _1(K), \Aff ^{\rm d}_1(K)\cap \BA _1(K)\rangle .
$$
Therefore,
Proposition~\ref{prop:gdaffine} (ii) holds when $n=1$.

Assume that $n=2$.
The following result is well known.

\begin{lemma}[{van der Kulk~\cite{Kulk}}]\label{lem:Kulk}
If $(f_1,f_2)\in \Aut _KK[x_1,x_2]\sm \Aff _2(K)$
is such that $\deg f_i\leq \deg f_j$
for $i,j\in \{ 1,2\} $ with $i\neq j$,
then
$\deg (f_j-\alpha f_i^{\rm d})<\deg f_j$ holds
for some $\alpha \in K\sm \zs $ and $d\geq 1$.
\end{lemma}

The following refinement of the lemma above
plays a crucial role in proving Proposition~\ref{prop:gdaffine} (ii)
in the case of $n=2$.

\begin{lemma}\label{lem:pKulk}
If $(f_1,f_2)\in \Aff ^{\rm d}_2(K)$
is such that $\deg f_i<\deg f_j$
for $i,j\in \{ 1,2\} $ with $i\neq j$,
then $\deg (f_j-\alpha f_i^{\rm d})<\deg f_j$
holds for some $\alpha \in K\sm \zs $ and
$d\in p\Z $ with $d>1$.
\end{lemma}
\begin{proof}\rm
By using Lemma~\ref{lem:Kulk} repeatedly,
we can get $\alpha _1,\ldots ,\alpha _d\in K$
with $\alpha _d\neq 0$
such that
\begin{equation}\label{eq:pKulk}
\deg (f_j-\alpha _df_i^{\rm d}-\alpha _{d-1}f_i^{d-1}
-\cdots -\alpha _2f_i^2)\leq \deg f_i.
\end{equation}
Then,
we have
$\deg (f_j-\alpha _df_i^{\rm d})<\deg f_j$
and $d=\deg f_j/\deg f_i\geq 2$,
since $\deg f_i<\deg f_j$ by assumption.
We show that $d$ is divisible by $p$
by contradiction.
Since $J(f_1,f_2)$ belongs to $\GL _2(K)$,
we have $\partial f_k/\partial x_l\in K$
for each $k$, $l$,
and $\partial f_i/\partial x_m\neq 0$
for some $m$.
Now,
suppose to the contrary that
$d$ is not divisible by $p$.
Then,
we have
$$
\deg
(d\alpha _df_i^{d-1}
+(d-1)\alpha _{d-1}f_i^{d-2}
+\cdots +2\alpha _2f_i)=(d-1)\deg f_i\geq \deg f_i.
$$
Put
$g=f_j-\alpha _df_i^{\rm d}-\alpha _{d-1}f_i^{d-1}
-\cdots -\alpha _2f_i^2$.
Since $\partial f_j/\partial x_m$
and $\partial f_i/\partial x_m$ belong to
$K$ and $K\sm \zs $, respectively,
it follows that
\begin{align*}
{\rm deg}{\partial g\over\partial x_m}
&=\deg \left(
{\partial f_j\over\partial x_m}
-( d\alpha _df_i^{d-1}
+(d-1)\alpha _{d-1}f_i^{d-2}
+\cdots +2\alpha _2f_i){\partial f_i\over\partial x_m}
\right) \\
&={\rm deg} ( d\alpha _df_i^{d-1}
+(d-1)\alpha _{d-1}f_i^{d-2}
+\cdots +2\alpha _2f_i)
\geq {\rm deg} f_i.
\end{align*}
On the other hand,
we have
$$
\deg{\partial g\over\partial x_m}
\leq \deg g-1\leq \deg f_i-1
$$
by (\ref{eq:pKulk}).
This is a contradiction.
Therefore,
$d$ is divisible by $p$.
\end{proof}

Now, let us complete the proof of
Proposition~\ref{prop:gdaffine} (ii) for $n=2$.
Thanks to Lemma~\ref{lem:compatibility},
we have only to check that
$\Aff ^{\rm d}_2(K)$ is equal to
$$
\ca ':=\langle
\Aff _2(K),
(\Aff ^{\rm d}_2(K)\cap \E ^1_2(K)),
(\Aff ^{\rm d}_2(K)\cap \E ^2_2(K))
\rangle .
$$
Clearly,
$\Aff ^{\rm d}_2(K)$ contains $\ca '$.
We prove the reverse inclusion by contradiction.
Suppose that there exists $F=(f_1,f_2)\in \Aff ^{\rm d}_2(K)$
not belonging to $\ca '$.
Without loss of generality,
we may assume that $\deg f_1\leq \deg f_2$,
and $\deg F:=\deg f_1+\deg f_2$
is minimal among such $F$'s.
Since $F$ does not belong to $\ca '$ by supposition,
$F$ does not belong to $\Aff _2(K)$.
Hence,
there exist $\alpha \in K\sm \zs $
and $d\geq 1$ with $d=1$ or $p\mid d$
such that
$\deg (f_2-\alpha f_1^{\rm d})<\deg f_2$
by Lemmas~\ref{lem:Kulk} and \ref{lem:pKulk}.
Then,
$$
E:=(x_1,x_2-\alpha x_1^{\rm d})
$$
belongs to $\Aff ^{\rm d}_2(K)\cap \E ^1_2(K)$,
and hence to $\ca '$.
Since $F$ is an element of $\Aff ^{\rm d}_2(K)\sm \ca '$,
it follows that
so is $F\circ E=(f_1,f_2-\alpha f_1^{\rm d})$.
On the other hand,
we have
$$
\deg F\circ E=\deg f_1+\deg (f_2-\alpha f_1^{\rm d})
<\deg f_1+\deg f_2=\deg F.
$$
This contradicts the minimality of $F$.
Thus, $\Aff ^{\rm d}_2(K)$ is contained in $\ca '$,
proving $\Aff ^{\rm d}_2(K)=\ca '$.
This completes the proof of Proposition~\ref{prop:gdaffine} (ii).

\section{Amalgamated product}

Throughout this section, ${\cal G}$ is a group and ${\cal A}$ and ${\cal B}$ are two subgroups of ${\cal G}$
such that ${\cal G}=\langle{\cal A},{\cal B}\rangle$.\\

{\bf A. Generalities}\\

A sequence $\al=(\al_1,\ldots,\al_l)$ of elements of ${\cal A}\cup{\cal B}$ is called an $({\cal A},{\cal B})$\textit{-word}.
The integer $l$ is called the \textit{length} of $\al$, denoted by ${\ell}(\al)$. The empty word is the only word which have length $0$.
When $l\ge 1$, we define ${\bf h}(\al)=\al_1$ and ${\bf t}(\al)=\al_l$.\\

Given two $({\cal A},{\cal B})$-words $\al=(\al_1,\ldots,\al_l)$ and $\beta=(\beta_1,\ldots,\beta_m)$,
we denote by $\al\beta=(\al_1,\ldots,\al_l,\beta_1,\ldots,\beta_m)$ the {\it concatenation} of $\alpha$ and $\beta$.
We denote by $W({\cal A},{\cal B})$ the monoid of all $({\cal A},{\cal B})$-words endowed with the concatenation and
by $\pi:W({\cal A},{\cal B})\to{\cal G}$ the canonical surjective homomorphism of monoids. Given an element $\gamma\in{\cal G}$,
any element in $\pi^{-1}(\gamma)$ is called an $({\cal A},{\cal B})$\textit{-decomposition} of $\gamma$.\\

Let $\al=(\al_1,\ldots,\al_l)$ and $\beta=(\beta_1,\ldots,\beta_m)$ be two $({\cal A},{\cal B})$-words, we define the following  relation:
$\al\thickapprox_{{\cal A}\cap{\cal B}}\beta$ if $l=m$ and there exist $\eta_1,\ldots,\eta_{l-1}\in{\cal A}\cap{\cal B}$ such that $\beta_1=\al_1\eta_1$, $\beta_i=\eta_{i-1}^{-1}\al_i\eta_i$ for all $i\in\{2,\ldots,l-1\}$ and $\beta_l=\eta_{l-1}^{-1}\al_l$.

\begin{remark}\label{equi}
The relation $\thickapprox_{{\cal A}\cap{\cal B}}$ is an equivalent relation. If $\al$ and $\beta$ are two words such that 
$\al\thickapprox_{{\cal A}\cap{\cal B}}\beta$ then $\pi(\al)=\pi(\beta)$.
\end{remark}

An $({\cal A},{\cal B})$-word $\al=(\al_1,\ldots,\al_l)$ is said to be \textit{reduced} if $l\ge 1$ and if $\al_i\in({\cal A}\pri{\cal B})\cup({\cal B}\pri{\cal A})$ for all $i\in\{1,\ldots,l\}$ and $\al_i\al_{i+1}\not\in{\cal A}\cup{\cal B}$ for all $i\in\{1,\ldots,l-1\}$.

\begin{remark}
a)  Let $\al$ be a reduced $({\cal A},{\cal B})$-word. Every $({\cal A},{\cal B})$-word $\beta$ such that
$\al\thickapprox_{{\cal A}\cap{\cal B}}\beta$ is also a reduced $({\cal A},{\cal B})$-word.\\
b) Given an element $\gamma\in{\cal G}\pri{\cal A}\cap{\cal B}$,
any element in $\pi^{-1}(\gamma)$ with a minimal length in $\pi^{-1}(\gamma)$ is a reduced $({\cal A},{\cal B})$-decomposition of $\gamma$.
\end{remark}

Given two reduced $({\cal A},{\cal B})$-words $\al=(\al_1,\ldots,\al_l)$ and $\beta=(\beta_1,\ldots,\beta_m)$,
we set $\delta=\delta(\al,\beta)={\bf t}(\al){\bf h}(\beta)=\al_l\beta_1$ and we define:
$$\al*\beta=\left\{\begin{array}{ccc}
                 \al\beta=(\al_1,\ldots,\al_l,\beta_1,\ldots,\beta_m)&{\rm if} & \delta\not\in{\cal A}\cup{\cal B}\\
                    (\al_1,\ldots,\al_{l-1},\al_l\beta_1,\beta_2\ldots,\beta_m) &{\rm if} &\delta\in{\cal A}\cup{\cal B}.
                   \end{array}
\right.$$

\begin{remark}\label{star} 
Assume that 
$\delta(\al,\beta)\not\in{\cal A}\cap{\cal B}$. 
Then, the following statements hold:\\
a) $\al*\beta$ is a reduced $({\cal A},{\cal B})$-word and
${\ell}(\al*\beta)\ge {\ell}(\al)+{\ell}(\beta)-1$,\\
b) if ${\ell}(\al)\ge 2$ then ${\bf h}(\al*\beta)={\bf h}(\al)$ and if ${\ell}(\beta)\ge 2$ then ${\bf t}(\al*\beta)={\bf t}(\beta)$.
\end{remark}

Since ${\cal G}=\langle {\cal A},{\cal B}\rangle $ 
by assumption, 
${\cal G}$ is the amalgamated product of 
${\cal A}$ and ${\cal B}$ over the intersection
${\cal A}\cap{\cal B}$ 
if and only if, 
for all $({\cal A},{\cal B})$-reduced words $\al$, 
we have $\pi(\al)\not=1$.
For a general definition of the amalgamated product see~\cite{Magnus}.

\begin{remark}\label{dequi}
This property is equivalent to say that if $\al$ and $\beta$ are two $({\cal A},{\cal B})$-reduced words such $\pi(\al)=\pi(\beta)$ then $\al\thickapprox_{{\cal A}\cap{\cal B}}\beta$.
\end{remark}

If ${\cal G}={\cal A}*_{\cap}{\cal B}$ then, given an element $\gamma\in{\cal G}\pri{\cal A}\cap{\cal B}$, 
all the $({\cal A},{\cal B})$-reduced decompositions of $\gamma$ have the same length and we call this length the $({\cal A},{\cal B})$\textit{-length of} $\gamma$ and we denote by ${\ell}(\gamma)$.
The $({\cal A},{\cal B})$-length of an element in ${\cal A}\cap{\cal B}$ is $0$.

\begin{remark}\label{redlg}
Let $\al$ be a reduced $({\cal A},{\cal B})$-decomposition of $\gamma\in{\cal G}\pri{\cal A}\cap{\cal B}$.
Let $\eta\in{\cal G}$ be an element of ${\cal G}$.\\
a) If $\eta{\bf h}(\al)\in{\cal A}\cap{\cal B}$ then ${\ell}(\eta\gamma)<{\ell}(\gamma)$.\\
b) If ${\bf t}(\al)\eta\in{\cal A}\cap{\cal B}$ then ${\ell}(\gamma\eta)<{\ell}(\gamma)$.\\
\end{remark}

{\bf B. Compatibility}\\

From now and until the end of this section, we assume that ${\cal G}={\cal A}*_{\cap}{\cal B}$.
We consider a subgroup ${\cal H}$ of ${\cal G}$ such that the following two properties are satisfied:
$$(i)\,\,{\cal A}\cap{\cal B}\subset{\cal H}\hspace{.3cm}{\rm and}\hspace{.3cm}
(ii)\,\,{\cal H}=\langle{\cal A}\cap{\cal H},{\cal B}\cap{\cal H}\rangle.$$
We call such a subgroup ${\cal H}$ an $({\cal A},{\cal B})$\textit{-compatible} subgroup of ${\cal G}$.\\

We consider also a subgroup $G$ of ${\cal G}$ and we use the following notation:
\begin{center}
$A={\cal A}\cap G$,\ \ $B={\cal B}\cap G$,\ \ $H={\cal H}\cap G$\ \ and\ \ $T=\langle A,B\rangle$.
\end{center}

\begin{remark}\label{ABred}
Any reduced $(A,B)$-word (resp.\ $({\cal A}\cap{\cal H},{\cal B}\cap{\cal H})$-word)
is also a reduced $({\cal A},{\cal B})$-word.
\end{remark}

\begin{proof} This follows from the inclusions $A\pri B\subset {\cal A}\pri{\cal B}$ and $B\pri A\subset {\cal B}\pri{\cal A}$.
\end{proof}

\begin{remark}\label{lgle2}
For all $g\in G\pri T$, we have $\ell(g)\ge 2$.
\end{remark}

\begin{proof} If $g\in G$ is such that $\ell(g)\le 1$ then $g\in G\cap({\cal A}\cup{\cal B})=A\cup B\subset T$.
\end{proof}

Let $\al=(\al_1,\ldots,\al_n)$ be a reduced $(H,T)$-word. We consider the following sets
$I_T(\al)=\{i\in\{1,\ldots,n\}\,;\,\al_i\in T\}$ and $I_H(\al)=\{i\in\{1,\ldots,n\}\,;\,\al_i\in H\}$.
For all $i\in\{1,\ldots,n\}$ we denote by $l_i={\ell}(\al_i)$ the $({\cal A},{\cal B})$-length of $\al_i$.
We define $L(\al)=\sum_{i\in I_H(\al)}l_i$.
We denote by ${\cal E}(\al)$ the set of all $(H,T)$-reduced words $\beta$ such that $\alpha\thickapprox_{H\cap T}\beta$.
We say that $\al$ is $({\cal A},{\cal B})$\textit{-minimal} if $L(\al)$ is minimal in the set
$L({\cal E}(\al))$.\\

Here is the main result of this section:

\begin{theorem}\label{thm}
Let $\al=(\al_1,\ldots,\al_n)$ be an $(H,T)$-reduced and $({\cal A},{\cal B})$-minimal word. 
For all $i\in I_T(\al)$ (resp.\ $i\in I_H(\al)$), let $D_i$ be a reduced 
$(A,B)$-decomposition (resp.\ a reduced $({\cal A}\cap{\cal H},{\cal B}\cap{\cal H})$-decomposition) 
of $\al_i$. 
Then, $D_1*\cdots*D_n$ is a reduced $({\cal A},{\cal B})$-decomposition of $\pi(\al)$.
\end{theorem}

\begin{proof} By Remark~\ref{ABred}, $D_i$ is also a reduced $({\cal A},{\cal B})$-word for all $i\in\{1,\ldots,n\}$.
We assume, by contradiction, that $D_1*\cdots*D_n$ is not a reduced $({\cal A},{\cal B})$-decomposi\-tion of $\pi(\al)$.
There exists $i\in\{1,\ldots,n-1\}$ such that $D:=D_1*\cdots*D_i$ is a reduced $({\cal A},{\cal B})$-decomposition of $\al_1\ldots\al_i$
but $D*D_{i+1}$ is not a reduced $({\cal A},{\cal B})$-decomposition of $\al_1\ldots\al_{i+1}$.
By Remark~\ref{star} a), we have: ${\bf t}(D){\bf h}(D_{i+1})\in{\cal A}\cap{\cal B}\,(\dag)$.\\
$\circ$ If ${\bf t}(D)={\bf t}(D_i)$ then, using $(\dag)$ and (i) of the compatibility assumption,
we have ${\bf t}(D_i){\bf h}(D_{i+1})\in{\cal A}\cap{\cal B}\subset{\cal H}$.
If $i\in I_T(\al)$ and $i+1\in I_H(\al)$ (resp.\ $i\in I_H(\al)$ and $i+1\in I_T(\al)$)
then $\eta:={\bf t}(D_i)\in H\cap T$ (resp.\ $\eta:={\bf h}(D_{i+1})^{-1}\in H\cap T$).
The $({\cal A},{\cal B})$-word $\beta=(\al_1,\ldots,\al_i\eta^{-1},\eta\al_{i+1},\ldots,\al_n)$ belongs to ${\cal E}(\al)$ and is such that $L(\beta)<L(\al)$ since we have ${\ell}(\eta\al_{i+1})<{\ell}(\al_{i+1})$ by Remark~\ref{redlg}~a)
(resp.\ ${\ell}(\al_i\eta)<{\ell}(\al_i)$ by Remark~\ref{redlg}~b)),
which contradict the $({\cal A},{\cal B})$-minimality of $\al$.\\
$\circ$ If ${\bf t}(D)\not={\bf t}(D_i)$ then trivially $i\ge 2$ and by Remark~\ref{star} b), $l_i=1$. Using Remark~\ref{lgle2}, this implies
$i\in I_T(\al)$ and $D_i=(\al_i)$ with $\al_i\in A\cup B\pri H$.
Hence $i-1\in I_H(\al)$ and ${\bf t}(D_1*\cdots*D_{i-1})={\bf t}(D_{i-1})$ and ${\bf t}(D)={\bf t}(D_{i-1})\al_i$.
Using $(\dag)$ and (i) of the compatibility assumption, we deduce ${\bf t}(D_{i-1})\al_i{\bf h}(D_{i+1})\in{\cal H}$ and then $\al_i\in H$ which is impossible.\\
\end{proof}

{\bf C. Consequences}\\

The following corollary gives the structure of the group $\langle H,T\rangle$.

\begin{corollary}\label{product}
We have $\langle H,T\rangle=H*_{\cap}T$.
\end{corollary}

\begin{proof} Let $\al=(\al_1,\ldots,\al_n)$ be an $(H,T)$-reduced word.
By Remark~\ref{equi}, the property we want to prove ($\pi(\al)=1$) is true for one element of ${\cal E}(\al)$
if and only if is true for all elements of ${\cal E}(\al)$. Hence, we can assume that $\al$ is $({\cal A},{\cal B})$-minimal.
For all $i\in I_T(\al)$ (resp.\ $i\in I_H(\al)$), 
let $D_i$ be a reduced
$(A,B)$-decomposition (resp.\ a reduced $({\cal A}\cap{\cal H},{\cal B}\cap{\cal H})$-decomposition).
By Theorem~\ref{thm}, $D_1*\cdots*D_n$ is a reduced $({\cal A},{\cal B})$-decomposition of $\pi(\al)$.
By Remark~\ref{star}~a) applied $n-1$ times,
$${\ell}(D_1*\cdots*D_n)\ge\sum_{i=1}^nl_i-(n-1)\ge 1.$$
We deduce $\pi(\al)\ne 1$ using ${\cal G}={\cal A}*_{\cap}{\cal B}$.\\
\end{proof}

The following corollary can be used to prove that an element $h\in G$ is not in $\langle H,T\rangle$
using a reduced $({\cal A},{\cal B})$-decomposition of $h$.

\begin{corollary}\label{criterion}
Let $h$ be an element of $\langle H,T\rangle$ and let $\theta=(\theta_1,\ldots,\theta_m)$ be a reduced $({\cal A},{\cal B})$-decomposition of $h$.
Let $i\in\{1,\ldots,n\}$ be such that $\theta_i\in{\cal B}$. 
Then:\\
1) If $i=1$ then $\theta_1\in B({\cal B}\cap{\cal H})$.\\
2) If $i=m$ then $\theta_m\in ({\cal B}\cap{\cal H})B$.\\
3) If $i\in\{2,\ldots,m-1\}$ then $\theta_i\in ({\cal B}\cap{\cal H})B({\cal B}\cap{\cal H})$.
\end{corollary}

\begin{proof} Let $\al=(\al_1,\ldots,\al_n)$ be an $(H,T)$-reduced and $({\cal A},{\cal B})$-minimal decomposition of $h$ and
let $A_i$ be as in Theorem~\ref{thm} for all $i\in\{1,\ldots,n\}$.
By Theorem~\ref{thm}, $D_1*\cdots*D_n\thickapprox_{{\cal A}\cap{\cal B}}\theta$ (see Remark~\ref{dequi}).
We denote $\theta'=(\theta'_1,\ldots,\theta'_m)=D_1*\cdots*D_n$.\\
1) We assume $i=1$. If $\theta'_1={\bf h}(D_1)$ then $\theta'_1\in B\cup({\cal B}\cap{\cal H})$.
If $\theta'_1\not={\bf h}(D_1)$ then $D_1=(\al_1)$ with $\al_1\in B$, ${\bf h}(A_2)\in{\cal B}\cap{\cal H}$
and $\theta'_1=\al_1{\bf h}(D_2)\in B({\cal B}\cap{\cal H})$. In both cases, $\theta'_1\in B({\cal B}\cap{\cal H})$
and then $\theta_1\in\theta_1'({\cal A}\cap{\cal B})\subset B({\cal B}\cap{\cal H})({\cal A}\cap{\cal B})=B({\cal B}\cap{\cal H})$
(using compatibility).\\
2) Using 1), we obtain 2) by changing $h$ to $h^{-1}$.\\
3) We assume $i\in \{2,\ldots,m-1\}$.
We prove $\theta_i'\in ({\cal B}\cap{\cal H})B$ by induction on $n$.
Since $\theta_i\in({\cal A}\cap{\cal B})\theta_i'({\cal A}\cap{\cal B})$ the result 3) follows.
If $n=1$ then $\theta'=D_1$ and $\theta'_i\in B\cup({\cal B}\cap{\cal H})$.
Let's assume the result for $n-1$. We set $D=D_1*\cdots*D_{n-1}$.
If $D*D_n=DD_n$ then
the result follows from induction hypothesis and $n=1$ case. We assume $D*D_n\not=DD_n$.
If $i<\ell(D)$ then $\theta_i\in ({\cal B}\cap{\cal H})B({\cal B}\cap{\cal H})$ by induction hypothesis
and if $i>\ell(D)$ then $\theta_i\in ({\cal B}\cap{\cal H})B({\cal B}\cap{\cal H})$ by the case $n=1$.
In the case $i=\ell(D)$, we have $\theta_i={\bf t}(D){\bf h}(D_n)$ and ${\bf t}(D)\in ({\cal B}\cap{\cal H})B$ by 2)
and ${\bf h}(D_n)\in B({\cal B}\cap{\cal H})$ by~1). We deduce $\theta_i\in({\cal B}\cap{\cal H})B({\cal B}\cap{\cal H})$.
\end{proof}

\begin{remark}\label{sym}
We can exchange the roles of ${\cal A}$ and ${\cal B}$ in Corollary~\ref{criterion}.
\end{remark}

The converse of Corollary~\ref{criterion} is not true in general. Nevertheless we have the following partial result
for elements of length $3$.

\begin{corollary}\label{criterionlg3}
We assume that $A\subset {\cal H}$. Let $h$ be an element of $G$ such that $\ell(h)=3$ and $h\in{\cal B}{\cal A}{\cal B}$.
Then $h\in\langle H,T\rangle$ if and only if $h\in B{\cal H}B$.
\end{corollary}

\begin{proof} Let $\theta=(\theta_1,\theta_2,\theta_3)$ (with $\theta_1,\theta_3\in{\cal B}$ and $\theta_2\in{\cal A}$)
be a reduced $({\cal A},{\cal B})$-decomposition of $h$. If $h\in\langle H,T\rangle$ then, by Corollary~\ref{criterion} (see also
Remark~\ref{sym}), we have $\theta_1\in B({\cal B}\cap{\cal H})$, $\theta_3\in ({\cal B}\cap{\cal H})B$ and
$\theta_2\in ({\cal A}\cap{\cal H})A({\cal A}\cap{\cal H})$. Using $A\subset {\cal H}$,
we deduce $h=\theta_1\theta_2\theta_3\in B{\cal H}B$. Conversely if $h\in B{\cal H}B$,
there exist $b_1,b_2\in B$ and $\rho\in{\cal H}$ such that $h=b_1\rho b_2$. Hence $\rho=b_1^{-1}hb_2^{-1}\in G$ and we deduce
$\rho\in H$ and $h\in\langle H,T\rangle$.
\end{proof}

The following corollary is useful to study the normality of the subgroup $\langle H,T\rangle$.

\begin{corollary}\label{normal}
Let $t$ (resp.\ $g$) be an element of $\langle H,T\rangle$ (resp.\ $G$) and let $\theta$ (resp.\ $\gamma$) be a reduced
$({\cal A},{\cal B})$-decomposition of $t$ (resp.\ $g$). We assume ${\bf h}(\theta)\in{\cal A}$, ${\bf t}(\theta)\in{\cal A}$,
${\bf h}(\gamma)\in{\cal B}\pri B({\cal B}\cap{\cal H})$ and ${\bf t}(\gamma)\in{\cal B}$. Then $gtg^{-1}\in G\pri\langle H,T\rangle$.
\end{corollary}

\begin{proof} We write $\gamma=(\gamma_1,\ldots_,\gamma_l)$ and we set $\gamma^{-1}=(\gamma_l^{-1},\ldots_,\gamma_1^{-1})$.
By Remark~\ref{star}, the assumptions ${\bf h}(\theta)\in{\cal A}$, ${\bf t}(\theta)\in{\cal A}$ and ${\bf t}(\gamma)\in{\cal B}$ and  imply that $\gamma*\theta*\gamma^{-1}=\gamma\theta\gamma^{-1}$ is an
$({\cal A},{\cal B})$-decomposition of $gtg^{-1}$. Since ${\bf h}(\gamma\theta\gamma^{-1})={\bf h}(\gamma)\in{\cal B}\pri B({\cal B}\cap{\cal H})$,
Corollary~\ref{criterion} gives $gtg^{-1}\in G\pri\langle H,T\rangle$.
\end{proof}

\section{Affine type subgroup}

Let $K$ be a field of characteristic $p\ge 0$.
We denote by  ${\cal G}={\rm Aut}_KK[x,y]$ the automorphism group of the $K$-algebra $K[x,y]$.
We denote by ${\cal A}={\rm Aff}_2(K)$ (resp.\ ${\cal B}={\rm BA}_2(K)$) the affine (resp.\ triangular) subgroup of ${\cal G}$.
We recall that ${\cal G}={\cal A}*_{\cap}{\cal B}$ (by the Jung-van der Kulk theorem).\\

{\bf A. $p$-stable subset}\\

For every subset $I$ of $\N_{\ge 2}$ ($\N_{\ge 2}$ is the set of integers greater or equal to $2$), we consider
the following subgroup of ${\cal B}$:
$${\cal B}^I:=\{(x+P(y),y)\,;\,P(y)\in \bigoplus_{i\in I} Ky^i\}\subset{\cal B}.$$

\begin{proposition}\label{pstable}
Let $I$ be a subset of $\N_{\ge 2}$. 
The following five conditions are equivalent.\\
$$\hspace{-2.6cm}i)\hspace{.3cm}({\cal A}\cap{\cal B}){\cal B}^I={\cal B}^I({\cal A}\cap{\cal B}).
\hspace{1cm}i')\hspace{.3cm}({\cal A}\cap{\cal B}){\cal B}^I\subset{\cal B}^I({\cal A}\cap{\cal B}).$$
$$\hspace{-3.6cm}ii)\hspace{.3cm}(\forall\,n\in I)(\forall\,a,b\in K)\,\,(ay+b)^n\in\bigoplus_{i\in I} Ky^i\oplus Ky\oplus K.$$
$$\hspace{-5.4cm}iii)\hspace{.3cm}(\forall\,n\in I)\,\,(y+1)^n\in\bigoplus_{i\in I} Ky^i\oplus Ky\oplus K.$$
$$\hspace{-2.5cm}iv)\hspace{.3cm}(\forall\,n\in I)(\forall\,k\in\N_{\ge 2})\,\,k\le n\Rightarrow\left[{n\choose k}=0{\rm\ mod\ }p{\rm\ or\ } k\in I\right].$$
\end{proposition}

\begin{proof} The equivalence $i)\,\Leftrightarrow\,i')$ holds since both 
${\cal A}\cap{\cal B}$ and ${\cal B}^I$ are subgroups
of ${\cal B}$.
The implication $ii)\,\Rightarrow\,iii)$ is obvious. Conversely, we assume $iii)$ and we consider $a,b\in K$ and $n\in I$.
If $b=0$ then $(ay+b)^n=a^ny^n\in Ky^n$. If $b\ne 0$ then, changing $y$ to ${a\over b}y$ in $iii)$, we have
$(ay+b)^n=b^n({a\over b}y+1)^n\in\bigoplus_{i\in I} Ky^i\oplus Ky\oplus K$. The equivalence $iii)\Leftrightarrow iv)$
is a consequence of the binomial theorem. \\
We assume $i')$ and we prove $iii)$. Let $n$ be in $I$. We have $\beta:=(x+y^n,y)\in{\cal B}^I$ and $t:=(x,y+1)\in{\cal A}\cap{\cal B}$.
By $i')$, we have
$$t\beta=(x,y+1)(x+y^n,y)=(x+(y+1)^n,y+1)\in({\cal A}\cap{\cal B}){\cal B}^I\subset{\cal B}^I({\cal A}\cap{\cal B}).$$
We deduce that there exist $P(y)\in\bigoplus_{i\in I} Ky^i$, $a,d\in K^*$, $b,c,e\in K$ such that
$$(x+(y+1)^n,y+1)=(x+P(y),y)(ax+by+c,dy+e)=(ax+aP(y)+by+c,dy+e)$$
and this implies $(y+1)^n\in\bigoplus_{i\in I} Ky^i\oplus Ky\oplus K$.\\
Conversely, we assume $ii)$ and we prove $i')$. We consider $\al\in{\cal A}\cap{\cal B}$ and $\beta\in{\cal B}^I$, we have:
$$\al\beta=(ax+by+c,dy+e)(x+P(y),y)=(ax+aP(dy+e)+by+c,dy+e)$$
where $a,d\in K^*$, $b,c,e\in K$ and $P(y)\in\bigoplus_{i\in I} Ky^i$.
By $ii)$, we have $aP(dy+e)=Q(y)+fy+g$ with $Q(y)\in\bigoplus_{i\in I} Ky^i$ and $f,g\in K$.
Hence, $\beta_1^{-1}\al\beta\in{\cal A}\cap{\cal B}$ where $\beta_1=(x+Q(y),y)\in{\cal B}^I$
and $\al\beta=\beta_1(\beta_1^{-1}\al\beta)\in{\cal B}^I({\cal A}\cap{\cal B})$.
\end{proof}

We say that $I\subset\N_{\ge 2}$ is $p$\textit{-stable} 
if the equivalent conditions in Proposition~\ref{pstable} are satisfied. 


\begin{remark}
1) If $(I_m)_{m\in\N}$ is a sequence of $p$-stable subsets of $\N_{\ge 2}$ then
$\bigcup_{m\in\N}I_m$ and $\bigcap_{m\in\N}I_m$ are $p$-stable.\\
2) The subsets $\emptyset$, $\{2,\ldots,k\}$ (where $k\in\N_{\ge 2}$) and $\N_{\ge 2}$ are $p$-stable.
All $0$-stable subsets of $\N_{\ge 2}$ have this form.\\
3) If $p\ne 0$, the subsets $p^n\{1,\ldots,k\}$ and $p^n\N_+$ (where $k,n\in\N_+$) are $p$-stable.\\
4) If $p\ne 0$, the subsets $\{p^n,p^{n+1}\}$ (where $n\in\N_+$) are $p$-stable.\end{remark}
\begin{proof}
1) This is clear using $iii)$ of Proposition~\ref{pstable}.\\
2) This is clear using $iv)$ of Proposition~\ref{pstable}.\\
3) This follows from the formula $(y+1)^{p^nk}=(y^{p^n}+1)^k$ in $K$ for all $k,n\in\N_+$
(using $iii)$ of Proposition~\ref{pstable}).\\
4) This follows from the formula $(y+1)^{p^n+1}=(y^{p^n}+1)(y+1)=y^{p^n+1}+y^{p^n}+y+1$
in $K$ for all $n\in\N_+$ (using $iii)$ of Proposition~\ref{pstable}).
\end{proof}

{\bf B. Affine type subgroup}\\

For every subset $I\subset\N_{\ge 2}$, we set ${\cal A}^I:=\langle {\cal A},{\cal B}^I\rangle$.
We say that a subgroup ${\cal H}$ of ${\cal G}$ is an \textit{affine type} subgroup of ${\cal G}$ if
${\cal H}={\cal A}^I$ for some $p$-stable subset $I\subset\N_{\ge 2}$.

\begin{proposition}\label{affinetype}
Let $I$ be a $p$-stable subset of $\N_{\ge 2}$. Then
${\cal B}\cap{\cal A}^I={\cal B}^I({\cal A}\cap{\cal B})$.
\end{proposition}
\begin{proof}
The inclusion ${\cal B}\cap{\cal A}^I\supset{\cal B}^I({\cal A}\cap{\cal B})$ is obvious. Conversely, we consider $\beta\in{\cal B}\cap{\cal A}^I$. Let $(\theta_1,\ldots,\theta_m)$ be an $({\cal A},{\cal B}^I)$-decomposition
of $\beta$ of minimal length. For all $i\in\{1,\ldots m\}$ such that
$\theta_i\in{\cal B}^I$, we have $\theta_i\not={\rm id}$ and thus $\theta_i\in{\cal B}\pri{\cal A}$ (since ${\cal A}\cap{\cal B}^I=\{{\rm id}\}$).
For all $i\in\{2,\ldots m-1\}$ such that $\theta_i\in{\cal A}$,
we have $\theta_i\not\in{\cal A}\cap{\cal B}$ and thus $\theta_i\in{\cal A}\pri{\cal B}$
(since ${\cal B}^I({\cal A}\cap{\cal B}){\cal B}^I={\cal B}^I({\cal A}\cap{\cal B})$ by the property $i)$ of $p$-stability).
If $\theta_1\in{\cal A}\cap{\cal B}$ (resp.\ $\theta_m\in{\cal A}\cap{\cal B}$), we change $\beta$ to
$\theta_1^{-1}\beta$ (resp.\ $\beta\theta_m^{-1}$) and $(\theta_1,\ldots,\theta_m)$ to
$(\theta_2,\ldots,\theta_m)$ (resp.\ $(\theta_1,\ldots,\theta_{m-1})$). Now, $(\theta_1,\ldots,\theta_m)$
is a reduced $({\cal A},{\cal B})$-decomposition equivalent to $(\beta)$. Using ${\cal G}={\cal A}*_{\cap}{\cal B}$,
we deduce $m=1$, $\theta_1\in{\cal B}^I$ and
$\beta\in({\cal A}\cap{\cal B}){\cal B}^I({\cal A}\cap{\cal B})={\cal B}^I({\cal A}\cap{\cal B})$.
\end{proof}

\begin{corollary}\label{incressing}
Let $I$ and $J$ be two $p$-stable subsets of $\N_{\ge 2}$. Then $I\subset J$ if and only if ${\cal A}^I\subset{\cal A}^{J}$.
\end{corollary}
\begin{proof}
If $I\subset J$ then ${\cal B}^I\subset{\cal B}^J$ and ${\cal A}^I\subset{\cal A}^{J}$. Conversely,
we assume ${\cal A}^I\subset {\cal A}^{J}$.
By Proposition~\ref{affinetype}, this implies ${\cal B}^I({\cal A}\cap{\cal B})\subset {\cal B}^J({\cal A}\cap{\cal B})$.
Let $n$ be an element in $I$. We have $(x+y^n,y)\in {\cal B}^I\subset{\cal B}^J({\cal A}\cap{\cal B})$.
There exist $P(y)\in\bigoplus_{i\in J} Ky^i$, $a,d\in K^*$, $b,c,e\in K$ such that:\\
$(x+y^n,y)=(x+P(y),y)(ax+by+c,dy+e)=(ax+aP(y)+by+c,dy+e).$\\
Since $n\ge 2$, we deduce $n\in J$.
\end{proof}

\begin{corollary}\label{compatible}
Let ${\cal H}$ be an affine type subgroup of ${\cal G}$. Then
${\cal H}=\langle {\cal A}, {\cal B}\cap{\cal H}\rangle$.\\
In other words, ${\cal A}\subset{\cal H}$ and
${\cal H}$ is $({\cal A},{\cal B})$-compatible.
\end{corollary}
\begin{proof}
Let $I$ be a $p$-stable subset of $\N_{\ge 2}$ such that ${\cal H}={\cal A}^I$.
By Proposition~\ref{affinetype},
$\langle{\cal A},{\cal B}\cap{\cal H}\rangle=\langle{\cal A},{\cal B}\cap{\cal A}^I\rangle=\langle{\cal A},{\cal B}^I({\cal A}\cap{\cal B})\rangle
=\langle {\cal A},{\cal B}^I\rangle={\cal A}^I={\cal H}$.
\end{proof}

\begin{proposition}
If a subgroup ${\cal H}$ of ${\cal G}$ is  such that ${\cal H}=\langle {\cal A},{\cal B}\cap{\cal H}\rangle$
and if there exists a $p$-stable subset $I\subset \N_{\ge 2}$ such that ${\cal B}\cap{\cal H}={\cal B}^I({\cal A}\cap{\cal B})$
then ${\cal H}={\cal A}^{I}$. In particular we have the following correspondence:
$$
\begin{array}{|c|c|c|c|c|c|}
\hline
  p{\rm -stable\ subset\ of\ }\N_{\ge 2} & I & \emptyset & \{p^j\,;\,j\in\N_+\} & p\N_+ & \N_{\ge 2} \\
\hline
{\rm\ affine\ type\ subgroup}   & {\cal A}^I & {\rm Aff}_2(K) & {\rm Aff}^{\rm g}_2(K) & {\rm Aff}^{\rm d}_2(K) & {\rm Aut}_KK[x,y]\\
  \hline
\end{array}
$$
\end{proposition}
\begin{proof}
We have ${\cal H}=\langle {\cal A},{\cal B}\cap{\cal H}\rangle=\langle {\cal A},{\cal B}^I({\cal A}\cap{\cal B})\rangle
=\langle{\cal A},{\cal B}^I\rangle={\cal A}^{I}$.\\
By Proposition~\ref{prop:gdaffine}, we have ${\cal H}=\langle {\cal A},{\cal B}\cap{\cal H}\rangle$ for
${\cal H}={\rm Aff}^{\rm g}_2(K)$ or ${\cal H}={\rm Aff}^{\rm d}_2(K)$. The property
${\cal B}\cap{\cal H}={\cal B}^I({\cal A}\cap{\cal B})$ for ${\cal H}={\rm Aff}^{\rm g}_2(K)$ and $I=\{p^j\,;\,j\in\N_+\}$
(resp.\ ${\cal H}={\rm Aff}^{\rm d}_2(K)$ and $I=p\N_+$) is clear from the definitions.
\end{proof}
\ \\
{\bf C. Tame type subgroup}\\

We consider, now, a domain $R$ of characteristic $p$ such that the field of fractions of $R$ is $K$.
We denote by $G={\rm Aut}_RR[x,y]$. We consider $G$ as a subgroup of ${\cal G}$ in the natural way.
We consider the tame subgroup $T=\langle A,B\rangle$ where $A={\cal A}\cap G$ and $B={\cal B}\cap G$.
For every subset $I\subset\N_{\ge 2}$, we set $A^I:={\cal A}^I\cap G$.
For each $\gamma \in G $, 
we denote by $\ell (\gamma )$ the $({\cal A},{\cal B})$-length of $\gamma $. 
In this context, 
Remark~\ref{lgle2} can be improved in the following way:

\begin{remark}\label{lgle3}
For all $\gamma\in G\pri T$, we have $\ell(\gamma)\ge 3$. 
\end{remark}

\begin{proof} Let $\gamma\in G$ be such that $\ell(\gamma)\le 2$.
Eventually changing $\gamma$ to $\gamma^{-1}$, we can assume that $\gamma=\al\beta$ with $\al\in{\cal A}$ and $\beta\in{\cal B}$.
Then $\gamma(y)$ is an affine coordinate of $R[x,y]$ and there exists $a\in A$ such that $a(y)=\gamma(y)$. Hence $a^{-1}\gamma(y)=y$
which implies (since $a^{-1}\gamma\in G$) that $a^{-1}\gamma=b\in B$ and $\gamma=ab\in T$.
\end{proof}

Corollary~\ref{compatible} and Corollary~\ref{product} give:

\begin{theorem}\label{thm:productaffine}
Let ${\cal H}$ be an affine type subgroup of ${\cal G}$. We set $H={\cal H}\cap G$. Then
$$\langle H,T\rangle=H*_{\cap}T.$$
\end{theorem}

Corollary~\ref{compatible} and Corollary~\ref{normal} give:

\begin{corollary}\label{cor:notnormal}
We assume that $R$ is not a field. Let $I$ and $J$ be two $p$-stable subsets of $\N_{\ge 2}$.
If $J\not\subset I$ then $\langle A^I,T\rangle$ is not a normal subgroup of $\langle A^J,T\rangle$.
\end{corollary}

\begin{proof}
Since $R$ is not a field there exists $a\in R\pri(\{0\}\cup R^*)$. Since $J\not\subset I$, there exists $n\in J\pri I$.
We consider the polynomials $P_1(y)=y+ay^n$ and $P_2(y)=y-ay^n$ we have $P_1(P_2(y))=y$ mod $a^2R[y]$.
We consider $\beta_1=(a^2x+P_1(y),y)\in{\cal B}\cap{\cal A}^J$, $\beta_2=(a^{-2}(P_2(y)-x),y)\in{\cal B}\cap{\cal A}^J$ and $\tau=(y,x)\in A$.
By a 
well-known fact 
(see for example~\cite{Edo&Venereau}), we have:
$$\sigma:=\beta_1\tau\beta_2=(a^{-2}(P_2(a^2x+P_1(y))-y)\,,\,a^2x+P_1(y))\in A^J.$$
By Corollary~\ref{compatible}, ${\cal H}:={\cal A}^I$ is $({\cal A},{\cal B})$-compatible, we can apply Corollary~\ref{normal} with
$H=A^I$, $t=\tau\in{\cal A}$, $\theta=(\tau)$, $g=\sigma\in A^J$ and $\gamma=(\beta_1,\tau,\beta_2)$,
to deduce $gtg^{-1}\not\in\langle T,A^I\rangle$ and
$\langle T,A^I\rangle$ is not a normal subgroup of $\langle T,A^J\rangle$.
We just need to check that $\beta_1\not\in B({\cal B}\cap{\cal H})$. If $\beta_1\in B({\cal B}\cap{\cal H})$
then $B\beta_1\cap({\cal B}\cap{\cal H})\not=\emptyset$. Since
$$B\beta_1=\{(a^2x+y+ay^n+a^2P(y),y)\,;\,P\in R[y]\}\hspace{.3cm}{\rm and}$$
\begin{align*}
&{\cal B}\cap{\cal H}={\cal A}^I\cap{\cal B}\\
&\quad =\{(ux+P(y),vy+w)\,;\,u,v\in K^*,w\in K,P(y)\in\bigoplus_{i\in I} Ky^i\oplus Ky\oplus K\},
\end{align*}
from $n\not\in I$ and $a\in R\pri(\{0\}\cup R^*)$, we deduce $B\beta_1\cap({\cal B}\cap{\cal H})=\emptyset$.
\end{proof}

\begin{theorem}\label{thm:notnormal}
We assume that $R$ is not a field.
Let ${\cal H}_1,{\cal H}_2$ be two affine type subgroups of ${\cal G}$ such that ${\cal H}_1\subsetneqq{\cal H}_2$.
We set $H_1={\cal H}_1\cap G$ and $H_2={\cal H}_2\cap G$. Then $\langle H_1,T\rangle$ is a proper subgroup of $\langle H_2,T\rangle$
which is not normal.
\end{theorem}

\begin{proof}
Let $I$ (resp.\ $J$) be a $p$-stable subset of $\N_{\ge 2}$ such that ${\cal H}_1={\cal A}^I$ (resp.\ ${\cal H}_2={\cal A}^J$).
By Corollary~\ref{incressing}, ${\cal H}_1\subsetneqq{\cal H}_2$ implies $I\subsetneqq J$
and we can conclude using Corollary~\ref{cor:notnormal}.
\end{proof}

{\bf D. The length~3 differentially tame automorphisms}\\

With the notations introduced in the previous subsection, we focus on the case $I=p\N_+$ and we set
${\cal H}={\cal A}^I={\rm Aff}^{\rm d}_2(K)$ and $H={\cal H}\cap G$.\\

Let $a\in R\pri\{0\}$ be a nonzero element and let $P(y),Q(y)\in R[y]$ be two polynomial such that $P(Q(y))=y$ mod $aR[y]$.
We consider $\beta_1=(ax+P(y),y)\in{\cal B}$, $\beta_2=(a^{-1}(Q(y)-x),y)\in{\cal B}$ and $\tau=(y,x)\in A$.
By a well-known fact 
(see for example~\cite{Edo&Venereau}), we have:
$$\sigma({a,P,Q})=\beta_1\tau\beta_2=(a^{-1}(Q(ax+P(y))-y)\,,\,ax+P(y))\in G.$$
Let $B_0$ be the subgroup of $B$ defined by $B_0=\{(x+S(y)\,;\,S(y)\in R[y]\}$. This subgroup is such that
$B=(A\cap B)B_0=(A\cap B)B_0$. The double class of $\sigma({a,P,Q})$ modulo $B_0$ is easy to compute:
$$B_0\sigma({a,P,Q})B_0=\{\sigma({a,P_1,Q_1})\,;\,P_1=P{\rm\ mod\ }aR[y]\,,\,Q_1=Q{\rm\ mod\ }aR[y] \}$$
We recall (see for example [EV]) that $\sigma({a,P,Q})\in T$ if and only if there exist two constants $b\in R^*$ and $c\in R$ such that $P(y)=by+c$ mod $aR[y]$,
in other words if $B_0\sigma({a,P,Q})B_0\cap A\not=\emptyset$.
On the other hand, we have
$$J(\sigma({a,P,Q}))=
\left(
  \begin{array}{cc}
    Q'(ax+P(y)) & a^{-1}(P'(y)Q'(ax+P(y))-1) \\
    a & P'(y)
  \end{array}
\right)$$
and we deduce that $\sigma({a,P,Q})\in H$ if and only if $P'(y),Q'(y)\in R^*$.
In this situation, using Corollary~\ref{criterion}, we can prove a characterisation for elements in $\langle H,T\rangle$:

\begin{corollary}\label{ex}
We have: $\sigma({a,P,Q})\in\langle H,T\rangle$ if and only if there exists $b\in R^*$ such that
$P'(y)=b{\rm\ mod\ }aR[y]$.
\end{corollary}

\begin{proof} To check if $\sigma({a,P,Q})\in\langle H,T\rangle$ or not, we can change $\sigma({a,P,Q})$ to any element
in the double class $B_0\sigma({a,P,Q})B_0$. 
In this way, 
we can assume that all the coefficients of $P$ and $Q$ which are
$0$ modulo $a$ are in fact $0$.\\
$\circ$ If there exists a constant $b\in R^*$ such that $P'(y)=b$ mod $aR[y]$, then
we can assume $P'(y)=b$. Since $P(Q(y))=y$, this implies $Q'(y)=b^{-1}$ mod $aR[y]$ and then $Q'(y)=b^{-1}$.
We deduce $\beta_1,\beta_2\in{\cal H}$ and $\sigma({a,P,Q})\in H$.\\
$\circ$ Conversely, we assume $\sigma({a,P,Q})\in\langle H,T\rangle$.
We apply Corollary~\ref{criterion} to $h=\sigma({a,P,Q})$ and $(\beta_1,\tau,\beta_2)$ and $i=1$,
we obtain $\beta_1\in B_0({\cal B}\cap{\cal H})$ and $B_0\beta_1\cap{\cal H}\not=\emptyset$.
That means there exist
$P_1,Q_1\in R[y]$ such that $P_1={\rm\ mod\ }aR[y]$,  $Q_1=Q{\rm\ mod\ }aR[y]$ and $\sigma({a,P_1,Q_1})\in H$
and then there exists $b\in R^*$ such that $P_1'(y)=b$ and $P'(y)=b{\rm\ mod\ }aR[y]$.
\end{proof}

\section{Acknowledgments}
\setcounter{equation}{0}

This work is partially supported
by the Grant-in-Aid for Scientific Research (B) 24340006,
and the Grant-in-Aid for Young Scientists (B) 24740022,
Japan Society for the Promotion of Science.

\bigskip 

\begin{flushleft}
E. Edo\\
ERIM, \\
University of New Caledonia\\
BP R4 - 98851\\
Noum\'ea CEDEX\\
New Caledonia \\
{\tt eric.edo@univ-nc.nc} \\

\bigskip 

S. Kuroda\\
Department of Mathematics and Information Sciences\\
Tokyo Metropolitan University\\
1-1 Minami-Osawa, Hachioji \\
Tokyo 192-0397, Japan \\
{\tt kuroda@tmu.ac.jp}
\end{flushleft}


\begin{thebibliography}{99}

\bibitem{Edo&Venereau} E. Edo, S. V\'{e}n\'{e}reau,
{\it Length 2 variables of} $A[x,y]$ {\it and transfer},
Polynomial automorphisms and related topics (Krak\`{o}w, 1999).
Ann. Polon. Math. {\bf 76} (2001), no. 1-2, 67--76.

\bibitem{van den Essen} A. van den Essen, {\textsl{Polynomial Automorphisms
and the Jacobian Conjecture, }}Birkhauser Verlag, Basel-Boston-Berlin (2000).

\bibitem{Jung}H.~Jung,
\"Uber ganze birationale Transformationen der Ebene,
J.\ Reine Angew.\ Math.\ {\bf 184} (1942), 161--174.

\bibitem{Kulk}W.~van der Kulk,
On polynomial rings in two variables,
Nieuw Arch.\ Wisk. (3) {\bf 1} (1953), 33--41.

\bibitem{Kuroda}S. Kuroda, 
Elementary reducibility of automorphisms of a vector group, 
Saitama Math. J. {\bf 29} (2012), 79--87. 


\bibitem{Magnus} W. Magnus, A. Karass, D. Solitar, {\it Combinatorial group theory},
Dover Publications, Inc., New York (1976).

\bibitem{Nagata}M.~Nagata,
On Automorphism Group of $k[x,y]$,
Lectures in Mathematics, Department of Mathematics,
Kyoto University, Vol.\ 5,
Kinokuniya Book-Store Co.\ Ltd., Tokyo, 1972.


\bibitem{SU}
I.~Shestakov and U.~Umirbaev,
The tame and the wild automorphisms of polynomial rings in three variables,
J.\ Amer.\ Math.\ Soc.\ {\bf 17} (2004), 197--227.


\bibitem{Tanaka}
T. Tanaka and H. Kaneta,
Automorphism groups of vector groups over
a field of positive characteristic, Hiroshima Math. J.
{\bf 38} (2008), 437--446.

\end{thebibliography}
\end{document}